\documentclass[12pt]{article}

\usepackage{amsfonts}
\usepackage{amssymb}
\usepackage{amsthm}
\usepackage{amsmath}
\usepackage{mathrsfs}
\usepackage{xcolor}
\usepackage[all]{xy}
\usepackage[utf8]{inputenc}

\def\uno{\mathbf{1}}

\usepackage[left=1.20in,right=1.20in]{geometry}
\definecolor{ao(english)}{rgb}{0.0, 0.5, 0.0}
\definecolor{brickred}{rgb}{0.8, 0.25, 0.33}
\definecolor{burntorange}{rgb}{0.8, 0.33, 0.0}
\definecolor{beaver}{rgb}{0.62, 0.51, 0.44}
\definecolor{brown(traditional)}{rgb}{0.59, 0.29, 0.0}
\definecolor{ao(english)}{rgb}{0.0, 0.5, 0.0}
\definecolor{verde}{rgb}{0.12, 0.8, 0.17}
\def\rouge{\color{red}}

\newcommand{\rac}{\mathbb Q}

\theoremstyle{plain}
\newtheorem{teo}{Theorem}[section]
\newtheorem*{teo-non}{Theorem}
\newtheorem{prop}[teo]{Proposition}
\newtheorem{coro}[teo]{Corollary}
\newtheorem{lema}[teo]{Lemma}

\theoremstyle{definition}
\newtheorem{defi}[teo]{Definition}

\theoremstyle{remark}
\newtheorem{ejem}[teo]{Example}

\newtheorem{rem}[teo]{Remark}

\def\endpf{~\leaders\hbox to 1em{\hss\  \hss}\hfill~\raisebox{.5ex}{\framebox[1ex]{}}\smallskip\par}

\reversemarginpar
\author{Nadia Romero\footnote{\texttt{nadia.romero@ugto.mx}}\\ 
\begin{small}
Departamento de Matem\'aticas,
\end{small}\\
\begin{small}
Universidad de Guanajuato.
\end{small}
}
\title{Hochschild cohomology for functors on linear symmetric monoidal categories}
\date{ }

\begin{document}

\maketitle
\begin{abstract}
Let $R$ be a commutative ring with unit. We develop a Hochschild cohomology theory in the category $\mathcal{F}$ of linear functors defined from an essentially small symmetric monoidal category enriched in $R$-Mod, to $R$-Mod. The category $\mathcal{F}$ is known to be symmetric monoidal too, so one can consider monoids in $\mathcal{F}$ and modules over these monoids, which allows for the possibility of a Hochschild cohomology theory. The emphasis of the article is in considering natural \textit{hom} constructions appearing  in this context. These \textit{homs}, together with the abelian structure of $\mathcal{F}$ lead to nice definitions and provide effective tools to prove the main properties and results of the classical Hochschild cohomology theory. 


\end{abstract}

\section{Introduction}

We consider  an essentially small symmetric monoidal category, $\mathcal{X}$, enriched in $R$-Mod with $R$ a commutative ring with unit, and then $R$-linear functors from $\mathcal{X}$ to $R$-Mod. This category of functors, denoted by $\mathcal{F}$, is an abelian, symmetric monoidal, closed category, via Day's convolution (see, for example, Section 3.3 of \cite{libro}). Given a monoid $A$ in $\mathcal{F}$, we have then a category a $A$-modules, $A$-Mod. We refer the reader to \cite{maclane} or \cite{libro} for the generalities on monoids and modules over them. The main example we have in mind when considering these hypothesis is $\mathcal{X}$ as the \textit{biset category} and $\mathcal{F}$ as the category of \textit{biset functors} (see, for example, \cite{biset}). Monoids in this case are called Green biset functors and have been extensively studied in the last years, in particular, in \cite{centros} the commutant and the center of a Green biset functor are studied. As we will see, the commutant of a Green biset functor is the Hochschild cohomology functor of degree zero. So, it is a natural question to ask for a Hochschild cohomology theory of Green biset functors. The purpose of the paper is to develop the main results and properties of this theory for monoids in $\mathcal{F}$, in order to apply them to biset functors in a forthcoming paper in collaboration with Serge Bouc.

Even though there are recent articles about Hochschild cohomology in monoidal categories (see for example \cite{italianos} and \cite{chico1}), none of them is suited for a direct application to our case.  So, we take a different and new approach by making use of the \textit{internal hom} functor in $\mathcal{F}$, defined in Section 3. The advantage of working with functors from $\mathcal{X}$ to $R$-Mod is that we can give an explicit construction of the internal hom functor (see chapters 2 and 3 of \cite{libro} for the abstract definition), via the Yoneda-Dress construction, introduced in Section 8.2 of \cite{biset}. With this, the definition of the Hochschild cochain complex of $A$ appears in a natural way in the category $\mathcal{F}$. In particular, we can give an explicit description of its arrows, in terms of these \textit{homs}, which allows for a better understanding of the Hochschild cohomology functors, denoted by $\mathcal{H}H^i(A,\, M)$, for $M$ an $A$-bimodule and $i$ a natural number.

The article is inspired in  Loday's  presentation of the classical Hochschild cohomo\-logy theory (see \cite{loday}). 
So, in  Section 4 we introduce the bar resolution of a monoid $A$ in $\mathcal{F}$ and we define the Hochschild cochain complex of $A$. In Section 5 we deal with the important example of a separable monoid. The main result is Theorem \ref{teosep}, in which we show that if $A$ is a separable monoid in $\mathcal{F}$, then, for any $A$-bimodule $M$, the  functors $\mathcal{H}H^i(A,\, M)$ are zero for very $i>0$ and, conversely, if $A$ is a monoid in $\mathcal{F}$ such that the first Hochschild cohomology functor $\mathcal{H}H^1(A,\, M)$ is zero for any $A$-bimodule $M$, then $A$ is separable. As a corollary, we obtain that the Hochschild cohomology functors, for $i> 0$, of the Burnside biset functor, $RB$, as well as those of the biset functor of rational representations, $kR_{\rac}$, over a field $k$ of characteristic $0$, are all zero.

Finally, in Section 6, we describe the Hochschild cohomology functors of degrees 0, 1 and 2. As in the classical case, the degree 1 can be described in terms of \textit{derivations} and the degree 2 is described in terms of extensions. Also, following \cite{gers}, for a monoid $A$ in $\mathcal{F}$, we endow the coproduct $\bigoplus_{i\in \mathbb{N}}\mathcal{H}H^i(A,\, A)$ with a structure of a graded monoid and call it the  \textit{Hochschild cohomology monoid}.  We try too keep the proofs of the statements in this last section as brief as possible, since there are many straightforward calculations that work as in the classical case.

\section{Preliminaries}
\label{sec-prem}

For an object $x$ in a category $\mathcal{C}$, we denote the identity morphism of $x$ simply as $x$.

In what follows $R$ is a commutative ring with identity, $R$-Mod denotes the category of $R$-modules and $(\mathcal{X},\, \diamond,\, \uno,\,\alpha,\, \lambda,\, \rho,\, s)$ is an essentially small symmetric monoidal category, with symmetry $s$ and enriched in $R$-Mod, in the sense of Definition 3.1.51 of \cite{libro}. In particular  the functor $\,\_\,\diamond \,\_\, : \mathcal{X}\times \mathcal{X}\rightarrow \mathcal{X}$ is $R$-bilinear. The category of $R$-linear functors from $\mathcal{X}$ to $R$-Mod is denoted by $\mathcal{F}$. Recall that, since $\mathcal{X}$ is preadditive and $R$-Mod is bicomplete and abelian, then $\mathcal{F}$ is a bicomplete and abelian category and all limits and colimits are computed pointwise (see, for example, Chapter 2 of Vol. 1 and Chapter 1 of Vol. 2 of \cite{borc}). Moreover,  $\mathcal{F}$ is a symmetric monoidal closed category with identity given by $I=\mathcal{X}(\uno,\,\_\,)$ (see, for example, Section 3.3 in \cite{libro}). We denote the tensor product of two objects $M$ and $N$ in $\mathcal{F}$ as $M\otimes N$. With this, the complete notation for $\mathcal{F}$ is $(\mathcal{F},\, \otimes,\, I,\, \alpha^{\mathcal{F}},\, \lambda^{\mathcal{F}},\, \rho^{\mathcal{F}},\, S)$. As it is customary, we will avoid, whenever is possible, the parenthesis of association in tensor product of several objects in $\mathcal{F}$.

Given a monoid $A$ in $\mathcal{F}$, we denote its product by $\mu:A\otimes A\rightarrow A$, in case of confusion we may write $\mu_A$ instead of $\mu$. We recall that we have a morphism $e:I\rightarrow A$, or $e_A$, and the following commuting diagrams,
\[
1.\quad \xymatrix{
A\otimes (A\otimes A)\ar[dd]^-{\cong}_-{\alpha^{\mathcal{F}}}\ar[r]^-{A\otimes \mu} & A\otimes A\ar[rd]^-{\mu}\\
 & & A\\
(A\otimes A)\otimes A\ar[r]^-{\mu\otimes A} & A\otimes A\ar[ru]_-{\mu}
}\quad
2.\quad \xymatrix{
I\otimes A\ar[r]^-{e\otimes A}\ar[rd]^-{\cong}_-{\lambda^{\mathcal{F}}} & A\otimes A\ar[d]^-{\mu} & A\otimes I\ar[l]_-{A\otimes e}\ar[ld]_-{\cong}^-{\rho^{\mathcal{F}}}\\
& A & 
} 
\]
in $\mathcal{F}$.

In what follows $A$ and $C$ are monoids in $\mathcal{F}$ (the letter $B$ will exclusively be used to denote the  Burnside functor). A morphism of monoids from $A$ to $C$ is an arrow $f:A\rightarrow C$ in $\mathcal{F}$ such that the following diagram commutes
\[
\xymatrix{
A\otimes A\ar[r]^-{\mu_A}\ar[d]_-{f\otimes f}&A\ar[d]^-{f}\\
C\otimes C\ar[r]_-{\mu_C}&C
}
\]
and $f\circ e_A=e_C$.

The tensor product $A\otimes C$, of $A$ and $C$, is a monoid with product given by
\[
\xymatrix{
(A\otimes C)\otimes (A\otimes C)\ar[rr]^-{A\otimes S_{2,\, 3}\otimes C}&& (A\otimes A)\otimes (C\otimes C)\ar[rr]^-{\mu_A\otimes \mu_C}&&A\otimes C,
}
\]
where $S_{2,\, 3}$ is the symmetry $C\otimes A\rightarrow A\otimes C$. When needed, as it was in this case, we will denote the symmetry with subindex. The identity arrow is given by $e_A\otimes e_C$, more precisely $e_{A\otimes C}$ is equal to the composition 
\[
\xymatrix{
I\ar[rr]^-{(\lambda^{\mathcal{F}})^{-1}} && I\otimes I\ar[rr]^-{e_A\otimes e_C} && A\otimes C.
}
\]

A left $A$-module is an object $M$ in $\mathcal{F}$ together with an arrow $\nu:A\otimes M\rightarrow M$ such that the following diagrams commute
\[
3.\quad \xymatrix{
A\otimes (A\otimes M)\ar[dd]^-{\cong}_-{\alpha^{\mathcal{F}}}\ar[r]^-{A\otimes \nu} & A\otimes M\ar[rd]^-{\nu}\\
 & & M\\
(A\otimes A)\otimes M\ar[r]^-{\mu\otimes M} & A\otimes M\ar[ru]_-{\nu}
}\quad
4.\quad \xymatrix{
I\otimes M\ar[r]^-{e\otimes M}\ar[rd]^-{\cong}_-{\lambda^{\mathcal{F}}} & A\otimes M\ar[d]^-{\nu} \\
& M & 
} 
\]
in $\mathcal{F}$. Right $A$-modules are defined in an analogous way. Morphisms of $A$-modules are defined in the obvious way and the subcategory of $\mathcal{F}$ of left $A$-modules is denoted by $A$-Mod.

We denote the opposite of $A$ in $\mathcal{F}$ as $A^{op}$. That is, $A^{op}$ denotes the object $A$ with opposite monoidal structure, $\mu \circ S:A\otimes A\rightarrow A$ and $e_{A^{op}}=e_A$. The monoid $A$ is called \textit{commutative} if it is equal to its opposite, i.e. $\mu\circ S=\mu$.

Recall that $C^{op}$-modules identify with right $C$-modules and thus that $(A,\, C)$-bimodules identify with $A\otimes C^{op}$-modules in the following way. 

\begin{rem}
\label{bimodop}
Let $M$ be an $(A,\, C)$-bimodule, with actions $l_A:A\otimes M\rightarrow M$ and $r_C:M\otimes C\rightarrow M$ such that $r_C\circ (l_A\otimes C)=l_A\circ (A\otimes r_C)$.  Then the action  of $A\otimes C^{op}$ on $M$ is given by
\[
\xymatrix{
(A\otimes C^{op})\otimes M\ar[rr]^-{A\otimes S_{2,\, 3}} && A\otimes M\otimes C\ar[rr]^-{l_A\otimes C} && M\otimes C\ar[r]^-{r_C}& M. 
}
\]
The symmetry is actually $S_{2,\, 3}:C\otimes M\rightarrow M\otimes C$, but in $A\otimes C^{op}$ we consider $C$ with the opposite product, as above. On the other direction, if $M$ is an $A\otimes C^{op}$-module we obtain the actions of $A$ and $C$ by composing with $A\otimes e_{C^{op}}:A\otimes I\rightarrow A\otimes C^{op}$ and $e_A\otimes C^{op}:I\otimes C^{op}\rightarrow A\otimes C^{op}$.
\end{rem}

We denote the monoid $A\otimes A^{op}$ by $A^e$. We may indistinctly write $A^e$-module, $(A,\, A)$-bimodule or even $A$-bimodule.

\begin{rem}
\label{mu-uni}
The objects $A$ and $A\otimes A$ have natural structures of $(A,\, A)$-bimodules. Furthermore, $\mu_A$ is a bimodule homomorphism. On the other hand, with the monoid structure defined before for $A\otimes A$, the arrow $\mu_A$ is not, in general, a morphism of monoids. Nevertheless, it is easy to see that $\mu_A$ is a unitary morphism, that is $e_A=\mu_A\circ e_{A\otimes A}$. 
\end{rem}

\begin{ejem}
Let $\mathcal{C}$ be the biset category for the class of all finite groups over the ring $R$, as defined in \cite{biset}, 
and let $\mathcal{F}_{\mathcal{C},\, R}$ be the category of biset functors, that is, $R$-linear functors from $\mathcal{C}$ to $R$-Mod. Then $\mathcal{C}$, with the direct product of groups and the trivial group, satisfies the hypothesis of $\mathcal{X}$ and so $\mathcal{F}_{
\mathcal{C}, R}$ those of $\mathcal{F}$. In this case, the functor $I$ is the Burnside functor $RB$.
\end{ejem}

\section{Hom functors}
\label{homfun}
We begin by recalling the construction of the \textit{internal hom} object in $\mathcal{F}$. The general definition of this object can be found in Section 3 of \cite{libro} but, since we are working with functors on $R$-Mod, we follow the lines of Chapter 8 of \cite{biset}, to give an explicit construction.

For each object $x\in \mathcal{X}$, we have an $R$-linear functor $\mathrm{p}_x:\mathcal{X}\rightarrow \mathcal{X}$, given by the monoidal structure of $\mathcal{X}$, that is, it sends an object $y$ to $y\diamond x$ and an arrow $\varphi$ to $\varphi\diamond x$. This  in turn allows us to define the endofunctor $\mathrm{P}_x:\mathcal{F}\rightarrow \mathcal{F}$ called the \textit{Yoneda-Dress construction}. In an object $F$, it is defined as $F_x:=F\circ\mathrm{p}_x$ and in an arrow $f:M\rightarrow N$ as $(P_x(f))_y=f_{y\diamond x}=:f_x(y)$, for $y$ an object of $\mathcal{X}$.

The following lemma is a generalization of Lemma 8.2.4 in \cite{biset}. Its proof is straightforward. 

\begin{lema}
\label{824}
Let $x$, $y$ and $x$ be objects in $\mathcal{X}$.
\begin{enumerate}
\item The functors $P_y\circ P_x$ and $P_{y\diamond x}$ are naturally isomorphic.
\item Given an arrow $\varphi$ in $\mathcal{X}(x,\, y)$, it induces a natural transformation $\mathrm{P}_{\varphi}:\mathrm{P}_x\rightarrow \mathrm{P}_y$. In a functor $F$, the arrow $\mathrm{P}_{\varphi,\, F}:F_x\rightarrow F_y$, denoted by   $F_{\varphi}$, is defined at an object $w$ in $\mathcal{X}$ as $F_{\varphi, w}=F(w\diamond \varphi)$. 
\item If $\phi\in \mathcal{X}(x,\, y)$ and $\psi\in \mathcal{X}(y,\, z)$, then $P_{\psi}\circ P_{\phi}=P_{\psi\circ \phi}$.
\item The correspondence $P_{\,\_\,}$ sending an object $x$ of $\mathcal{X}$ to $P_{x}$ and an arrow $\varphi$
 in $\mathcal{X}$ to $P_{\varphi}$ is an $R$-linear functor from $\mathcal{X}$ to the category $Fun_{R}(\mathcal{F},\, \mathcal{F})$ of $R$-linear endofunctors of $\mathcal{F}$.
\end{enumerate}
\end{lema}

By point 4 of the previous lemma, for every object $F$ in $\mathcal{F}$, we have an $R$-linear  functor $F_{\,\_\,}:\mathcal{X}\rightarrow \mathcal{F}$.

\begin{rem}
For the biset category, the Yoneda-Dress construction $P_G$ at a group $G$, satisfies also that $P_G$ is a self-adjoint functor. We will not need this extra condition in what follows.
\end{rem}

\begin{defi}
\label{internal}
Let $M$ and $N$ be objects in $\mathcal{F}$, we denote by $\mathcal{H}(M,\, N)$ the object in $\mathcal{F}$ defined by the following composition of functors
\begin{displaymath}
\mathcal{H}(M,\, N)=\textrm{Hom}_{\mathcal{F}}(M,\, N_{\,\_\,}).
\end{displaymath}
The correspondence $\mathcal{H}(\,\_\,,\,\_\,):\mathcal{F}^{op}\times \mathcal{F}\rightarrow \mathcal{F}$ is a bilinear functor, called the internal hom functor. 
\end{defi}

\begin{lema}
\label{yonF}
We have a Yoneda Lemma for the internal hom functor. If $x\in \mathcal{X}$, then 
\begin{displaymath}
\mathcal{H}(\mathcal{X}(x,\,\_\,),\, F)\cong F_x.
\end{displaymath}
\end{lema}
\begin{proof}
The known isomorphism of $R$-modules, $\mathrm{Hom}_{\mathcal{F}}(\mathcal{X}(x,\,\_\,),\, F)\leftrightarrow F(x)$ defines a natural equivalence between the functor $\mathrm{Hom}_{\mathcal{F}}(\mathcal{X}(x,\,\_\,),\,\_\,)$ and the evaluation functor, $ev_x$, going from $\mathcal{F}$ to $R$-Mod. By precomposition with the functor $F_{\,\_\,}$ we obtain an equivalence between $\mathcal{H}(\mathcal{X}(x,\,\_\,),\,F)$ and $ev_x\circ F_{\,\_\,}$ in $\mathcal{F}$. Finally, by composing with the symmetry $s$ of $\mathcal{X}$, we obtain an equivalence between $ev_x\circ F_{\,\_\,}$ and $F_x$.
\end{proof}

The rest of the section is devoted to defining a hom functor for modules over monoids in $\mathcal{F}$. We begin by stating the universal property of tensor products. It is easy to see that the proof of Proposition 8.4.2 in \cite{biset} holds in our case, so Remark 8.4.3 in \cite{biset} can also be generalized.

\begin{rem}
\label{flechas}
Let $M$, $N$ and $T$ be objects in $\mathcal{F}$. Consider the functors $(M,\, N)$ and $T$ from $\mathcal{X}\times \mathcal{X}$ to $R$-Mod given by sending $(x,\, y)$ to $M(x)\otimes_R N(y)$ and $T(x\diamond y)$ respectively, and in an obvious way in arrows. Then, there is a one-to-one correspondence between Hom$_\mathcal{F}(M\otimes N,\, T)$ and  $\mathrm{Hom}_{\mathcal{F}}((M,\, N), T)$ which is natural in every variable. 

In other words, a natural transformation from $M\otimes N$ to $T$ is given by a collection of $R$-bilinear maps
\begin{displaymath}
M(x)\times N(y)\rightarrow T(x\diamond y)
\end{displaymath}
satisfying obvious functoriality conditions. 
\end{rem}

Notice that this remark can be generalized to an arrow from a finite product $\bigotimes_{i=1}^nM_i$ to $T$.

Let $A$ be a monoid in $\mathcal{F}$. The Yoneda Lemma,
\begin{displaymath}
\mathrm{Hom}_{\mathcal{F}}(\mathcal{X}(\uno,\,\_\,),\, A)\cong A(\uno)
\end{displaymath}
together with the previous remark allows us to see the monoid $A$ as in the paragraphs following Definition 8.5.1 of \cite{biset}. 

Let $x$, $y$ and $z$ be objects in $\mathcal{X}$, consider $\alpha_{x,\,y,\,z}:x\diamond(y\diamond z)\rightarrow (x\diamond y)\diamond z$, $\lambda_x:\uno\diamond x\rightarrow x$ and $\rho_x:x\diamond \uno\rightarrow x$. Then, $A$ is a monoid in $\mathcal{F}$ if for every pair $x,\, y$ of objects in $\mathcal{X}$, there exist bilinear maps 
\begin{displaymath}
A(x)\times A(y)\rightarrow A(x\diamond y),
\end{displaymath}
denoted simply by $(a,\, b)\mapsto a\times b$, and an element $\varepsilon_A\in A(\uno)$ satisfying the following conditions. 
\begin{itemize}
\item Associativity. For $x$, $y$ and $z$, objects in $\mathcal{X}$, and for any $a\in A(x)$, $b\in A(y)$ and $c\in A(z)$,
\begin{displaymath}
(a\times b )\times c =A(\alpha_{x,\, y,\, z})(a\times (b\times c)).
\end{displaymath}
\item Identity element. For an object $x$ in $\mathcal{X}$ and any $a\in A(x)$,
\begin{displaymath}
a=A(\lambda_x)(\varepsilon_A\times a)=A(\rho_x)(a\times \varepsilon_A).
\end{displaymath}
\item Functoriality. Let $\varphi:x\rightarrow x'$ and $\psi: y\rightarrow y'$ be arrows in $\mathcal{X}$. Then for any $a\in A(x)$ and $b\in A(y)$,
\begin{displaymath}
A(\varphi\diamond \psi)(a\times b)=A(\varphi)(a)\times A(\psi)(b).
\end{displaymath}
\end{itemize}

Also, a module over $A$ can be seen as in the paragraphs following Definition 8.5.5 of \cite{biset}. That is, a left $A$-module is an object $M$ in $\mathcal{F}$ such that for every pair $x,\, y$ of objects in $\mathcal{X}$, there exist bilinear maps 
\begin{displaymath}
A(x)\times M(y)\rightarrow M(x\diamond y),
\end{displaymath}
which we continue to denote by $(a,\, m)\mapsto a\times m$, satisfying corresponding conditions of associativity, identity element (on the left only) and functoriality.  In a similar fashion we can define right modules and bimodules.

Morphisms of $A$-modules can also be rewritten in these terms. Let $M$ and $N$ be $A$-modules, a morphism of $A$-modules from $M$ to $N$ is an arrow $f:M\rightarrow N$ in $\mathcal{F}$ such that for any objects $x$ and $y$ in $\mathcal{X}$, we have $f_{x\diamond y}(a\times m)=a\times f_y(m)$ for all $a\in A(x)$ and $m\in M(y)$. 

The proof of the following lemma is straightforward.

\begin{lema}
\label{yonA1}
Let $A$ be a monoid in $\mathcal{F}$ and $M$ be an $A$-module, then we have an isomorphism of $R$-modules,
\begin{displaymath}
\mathrm{Hom}_{A-\mathrm{Mod}}(A,\,M)\cong M(\uno),
\end{displaymath}
given by sending a morphism of $A$-modules, $f$, to $f_{\uno}(\varepsilon_A)$.
\end{lema}

\begin{lema}
\label{restric}
Let $x$ and $y$ be objects in $\mathcal{X}$ and $\alpha:x\rightarrow y$ be an arrow. Consider $\mathrm{P}_x$ and $\mathrm{P}_{\alpha}$ as defined at the beginning of the section. Then $\mathrm{P}_{\,\_\,}$ defines a functor from $\mathcal{X}$ to the category of $R$-linear endofunctors of $A$-$\mathrm{Mod}$.
\end{lema}
\begin{proof}
Let $x$ and $\varphi$ be an object and an arrow in $\mathcal{X}$. By Lemma \ref{824}, we only need to verify that  $\mathrm{P}_x$  restricts from $A$-Mod to $A$-Mod and that the components of the transformation $\mathrm{P}_{\varphi}:\mathrm{P}_x\rightarrow \mathrm{P}_y$, as defined before, are morphisms of $A$-modules. 

Let $M$ be an $A$-module, then
\begin{displaymath}
A(y)\times M_x(z)\rightarrow M_x(y\diamond z),
\end{displaymath}
given by the original action of $A$ on $M$, endows $M_x$ with a structure of $A$-module. Also, if $f:M\rightarrow N$ is a morphism of $A$-modules, then it is immediate to see that $P_x(f)$, is a morphism of $A$-modules. Hence we can restrict $P_x:A$-$\textrm{Mod}\rightarrow A$-$\textrm{Mod}$. Now let $M$ be an $A$-module, the functoriality of the action of $A$ on $M$ shows that $\textrm{P}_{\alpha,\, M}:M_x\rightarrow M_y$ is a morphism of $A$-modules.
\end{proof}

\begin{rem}
\label{modder}
As said in Section \ref{sec-prem}, $N$ is a right $A$-module if and only if it is an $A^{op}$-module. The two structures are related by the following equation
\begin{displaymath}
a\times n =N(s_{z\diamond y}^{y\diamond z})(n\times a),
\end{displaymath}
for $a\in A(y)$,  $n\in N(z)$ and  where $s_{z\diamond y}^{y\diamond z}:z\diamond y\rightarrow y\diamond z$ is the symmetry in $\mathcal{X}$. Now, if we begin with a right $A$-module, $M$, then it is an $A^{op}$-module and then $N=M_x$ is an $A^{op}$-module, as in the previous lemma.  After some straightforward computations we see that its structure of right $A$-module is given by
\begin{displaymath}
M_x(z)\times A(y)\rightarrow M_x(z\diamond y),\, (m,\, a)\mapsto M(s_{z\diamond x\diamond y}^{z\diamond y\diamond x})(m\times a). 
\end{displaymath}
A word of caution is needed, since, in general we will write $m\times a$ for the action of $A$ in \textit{any} right $A$-module $N$ regardless of whether it is shifted or not. However, if $N$ is  explicitly a shifted module $M_x$, we will write the corresponding symmetry.
\end{rem}

Given two $A$-modules $M$ and $N$, we have that $\textrm{Hom}_{A-\mathrm{Mod}}(M,\,N)$ is an $R$-submodule of $\textrm{Hom}_{\mathcal{F}}(M,\,N)$. Moreover, for an $A$-module $M$, we have a functor $\textrm{Hom}_{A-\mathrm{Mod}}(M,\,\_):A$-$\mathrm{Mod}\rightarrow R$-$\mathrm{Mod}$. 
 
 \begin{defi}
Let $M$ and $N$ be objects in $A$-Mod, we denote by $\mathcal{H}_A(M,\, N)$ the object in $\mathcal{F}$ defined by
\begin{displaymath}
\mathcal{H}_A(M,\, N)=\textrm{Hom}_{A-\mathrm{Mod}}(M,\, N_{\,\_\,}).
\end{displaymath}
\end{defi}

\begin{lema}
\label{modA}
The correspondence $\mathcal{H}_A(\,\_\,,\,\_\,):(A$-$\mathrm{Mod})^{op}\times A$-$\mathrm{Mod}\rightarrow \mathcal{F}$ is a bilinear functor.
\end{lema}
\begin{proof}
This follows from Lemma \ref{restric}.
\end{proof}

\begin{coro}
\label{yonA}
For $A$ and $M$ as before, we have an isomorphism in $\mathcal{F}$,
\begin{displaymath}
\mathcal{H}_A(A,\,M)\cong M.
\end{displaymath}
\end{coro}
\begin{proof}
The bijection of Lemma \ref{yonA1} defines a natural equivalence between the functor $\mathrm{Hom}_{A-\mathrm{Mod}}(A,\,\_\,)$ and the evaluation functor, $ev^A_{\uno}$, going from $A$-Mod to $R$-Mod. By precomposing with $M_{\,\_\,}$
we obtain an equivalence between $\mathcal{H}_A(A,\, M)$ and $ev^A_{\uno}\circ M_{\,\_\,}$, but clearly the last one is equivalent to $M$.

\end{proof}

\section{The Hochschild cochain complex}

\subsection{Bar resolution of a monoid}
\label{subsec-resbar}

Let $A$ be a monoid in $\mathcal{F}$. We omit the morphism $\alpha$ of diagram 1 and write $A^{\otimes n}$ for the tensor product of $A$ with itself $n$ times. Also, for simplicity, we write $A^{\otimes i+j}$ instead of $A^{\otimes (i+j)}$, for suitable $i$ and $j$.
 
Consider the following sequence in $\mathcal{F}$,
\[
C_{\ast}^{bar}(A)  : \xymatrix{
\ldots\ar[r]&A^{\otimes n+1}\ar[r]^b&A^{\otimes n}\ar[r]^b&\ldots\ar[r]^b&A^{\otimes 2}
}
\]
where $A^{\otimes 2}$ is in degree zero and $b:A^{\otimes n+1}\rightarrow A^{\otimes n}$, for $n\geq 2$, is defined in the following way. Given $n\geq 2$,  define $d_i:A^{\otimes n+1}\rightarrow A^{\otimes n}$, for $1\leq i\leq n$, as 
\begin{displaymath}
d_i=A\otimes \cdots \otimes A\otimes \mu\otimes A\otimes\cdots \otimes A,
\end{displaymath} 
where $\mu$ appears in the position $i$ of the tensor product. That is, $d_1=\mu\otimes A^{\otimes n-1}$, $d_i=A^{\otimes i-1}\otimes \mu \otimes A^{\otimes n-i}$ if $1<i<n$ and $d_n=A^{\otimes n-1}\otimes \mu$. With this, $b:=\sum_{i=1}^n (-1)^{i+1}d_i$. When considering the augmented complex, having $A$ in degree -1, with morphism $b=\mu: A^{\otimes 2}\rightarrow A$ and 0 in all degrees $j<-1$, we will write $C_{\ast}^{bar}(A)_{\mu}$.

\begin{lema}
\label{presimplicial}
The sequence $C_{\ast}^{bar}(A)$ is a \textit{presimplicial} object in $\mathcal{F}$, that is, the arrows $d_i$ defined above satisfy
\begin{displaymath}
d_i\circ d_j =d_{j-1}\circ d_i \quad 1\leq i < j\leq n. 
\end{displaymath}
Moreover, $C_{\ast}^{bar}(A)_{\mu}$ is a complex in $\mathcal{F}$.
\end{lema} 
\begin{proof}
Notice that if $i+1<j$, then we clearly have $d_i\circ d_j =d_{j-1}\circ d_i$. On the other hand, the commutativity of diagram 1 implies the previous equality if $j=i+1$. It also implies that $\mu\circ b=0$. With this, the proof of $b\circ b=0$ is standard (see for example Lemma 1.0.7 of \cite{loday}).
\end{proof}

The structure of $A$-bimodule of $A$ is given by diagram 1. 
Now consider $A^{\otimes n}$ with $n\geq 2$. The structure of  $A$-bimodule of $A^{\otimes n}$ is given by $d_{n}\circ d_1:A^{\otimes n+2}\rightarrow A^{\otimes n}$, which is equal to $d_1\circ d_{n+1}$.

\begin{prop}
\label{laresbar}
Let $A$ be a monoid in $\mathcal{F}$, the complex $C_{\ast}^{bar}(A)$ is a resolution of the $A^e$-module $A$. It is called the bar resolution of $A$.
\end{prop}
\begin{proof}
We will first show that the complex $C_{\ast}^{bar}(A)_{\mu}$ is exact and then we will verify that it is a complex of $A^e$-modules.

Using diagram 1, one can show that $\mu$ is an epimorphism in $\mathcal{F}$ but we prefer to define the operators of \textit{extra degeneracy} and show that the complex is acyclic. 
Consider first $q:A\rightarrow A\otimes A$ as the morphism
\[
\xymatrix{
A\ar[r]^-{\lambda^{-1}}&I\otimes A\ar[r]^-{e\otimes A}&A\otimes A
}
\]
from diagram 2. Next, for $n\geq 2$, define
\begin{displaymath}
q:A^{\otimes n}\rightarrow A^{\otimes n+1},\quad q=((e\otimes A)\circ \lambda^{-1})\otimes A^{\otimes n-1}.
\end{displaymath}
 It is easy to see that if $2\leq i\leq n$, then $d_i\circ q=q\circ d_{i-1}$. On the other hand, diagram 2 shows that for any $n\geq 1$ we have $d_1\circ q=A^{\otimes n}$. This implies $b\circ q+q\circ b=A^{\otimes n}$ for $n\geq 2$ and $\mu\circ q=A$. Hence $q$ is a contracting homotopy and $C_{\ast}^{bar}(A)_{\mu}$ is acyclic.

Now we see that $C_{\ast}^{bar}(A)_{\mu}$ is a complex of $A^e$-modules. The arrow $\mu$ is a morphism of $A^e$-modules thanks to the  associativity of the product, i.e. diagram 1. Next, let $n\geq 2$ and $d_i:A^{\otimes n+1}\rightarrow A^{\otimes n}$ with $1\leq i\leq n$. To see that the action of $A^e$ commutes with $b$, we verify it commutes with the arrows $d_i$, that is, we check the commutativity of the following diagram
\[
\xymatrix{
A^{\otimes n+3}\ar[rr]^-{d_{n+1}\circ\, d_1}\ar[d]_-{A\otimes d_{i+1}\otimes A}&& A^{\otimes n+1}\ar[d]^-{d_i}\\
A^{\otimes n+2}\ar[rr]^-{d_n\circ\, d_1}&&A^{\otimes n},
}
\]
where the horizontal arrows correspond to the  structures of $A$-bimodules of $A^{\otimes n+1}$ and $A^{\otimes n}$. Notice that $A\otimes d_{i+1}\otimes A=d_{i+1}$. By Lemma \ref{presimplicial} we have
\begin{displaymath}
d_i\circ d_{n+1}\circ d_1=d_n\circ d_1\circ d_{i+1}.
\end{displaymath}
Hence, $C_{\ast}^{bar}(A)_{\mu}$ is a complex of $A^e$-modules. 
\end{proof}
 
\subsection{Hochschild cochain complex} 
\label{subsec-comple}

\begin{defi}
\label{defhoch}
Let $A$ be a monoid in $\mathcal{F}$ and $M$ be an $A$-bimodule. The Hochschild cochain complex is the complex in $\mathcal{F}$ given by $\mathcal{H}_{A^e}(C^{bar}_{\ast}(A),\, M)$. Since $\mathcal{F}$ is an abelian category, we define the Hochschild cohomology of $A$ with coefficients in $M$ as the object in $\mathcal{F}$ given by 
\begin{displaymath}
\mathcal{H}H^n(A,\, M):=H^n(\mathcal{H}_{A^e}(C^{bar}_{\ast}(A),\, M))
\end{displaymath}
\end{defi}

It is easy to see that $\mathcal{H}H^n(A,\,\_\,)$ defines a functor from $A^e$-Mod to $\mathcal{F}$. 

The next definition generalizes Definition 18 in \cite{centros}.

\begin{defi}
Let $A$ be a monoid in $\mathcal{F}$ and $M$ be an $A$-bimodule. The commutant of $M$ at $x\in \mathcal{X}$ is
\begin{displaymath}
CM_A(x)=\{m\in M(x)\mid a\times m=M(s_{x\diamond y}^{y\diamond x})(m\times a), \forall a\in A(y), \forall y\in \mathcal{X}\},
\end{displaymath} 
where $s_{x\diamond y}^{y\diamond x}:x\diamond y\rightarrow y\diamond x$ is the symmetry in $\mathcal{X}$. If the $A$-bimodule is $A$ itself, we keep the notation $CA$ of \cite{centros}.
\end{defi}

Since $s_{y\diamond x}^{x\diamond y}$ is a natural isomorphism, it is easy to see that $CM_A$ is a subfunctor of $M$. Also, we will see in the last section that the functors $\mathcal{H}H^n(A,\, M)$ are all $CA$-modules.

To describe the Hochschild cochain complex we need the following lemma.

\begin{lema}
\label{deschom}
Let $A$ be a monoid in $\mathcal{F}$, $M$ be an $A$-bimodule and $n\geq 0$ be an integer. We have the  following isomorphism in $\mathcal{F}$,
\begin{displaymath}
\mathcal{H}_{A^e}(A^{\otimes n+2},\, M)\cong \mathcal{H}(A^{\otimes n},\, M),
\end{displaymath}
where, if $n=0$, then $A^{\otimes 0}$ denotes $\mathcal{X}(\uno,\,\_\,)$. Also, $\mathcal{H}_{A^e}(A,\, M)\cong CM_A$ in $\mathcal{F}$.
\end{lema}
\begin{proof}
For the first statement, consider $n=0$. Using Remark \ref{bimodop}, it is easy to see that $A^{\otimes 2}$ and $A^e$ are isomorphic as $A^e$-modules. Then, by  lemmas \ref{yonF} and \ref{yonA}, we have the following isomorphisms in $\mathcal{F}$,
\begin{displaymath}
\mathcal{H}_{A^e}(A^{\otimes 2},\,M)\cong M \cong\mathcal{H}(\mathcal{X}(\uno,\,\_\,),\,M).
\end{displaymath}
Hence $\mathcal{H}_{A^e}(A^{\otimes 2},\, M)\cong \mathcal{H}(A^{\otimes 0},\, M)$.

Now suppose $n\geq 1$ and let $N$ be an $A$-bimodule. Then, by Remark \ref{flechas},  an arrow $f:A^{\otimes n+2}\rightarrow N$, of $A^e$-modules, is given by a corresponding  collection of maps
\begin{displaymath}
\tilde{f}_{w,\,y,\,z}:A(w)\times A^{\otimes n}(y)\times A(z)\rightarrow N(w\diamond y\diamond z)
\end{displaymath}
linear and functorial in each entry. These maps also satisfy the following identity
\begin{displaymath}
\tilde{f}_{u\diamond w,\, y,\, z\diamond v}(\alpha\times a,\, c,\, b\times \beta)=\alpha\times \tilde{f}_{w,\, y,\, z}(a,\, c,\, b)\times \beta
\end{displaymath}
for all objects $u$, $v$, $w$, $y$ and $z$ and elements $a$, $b$, $c$, $\alpha$ and $\beta$ in the corresponding evaluations. In particular 
\begin{displaymath}
\tilde{f}_{w,\, y,\, z}(a,\, c,\, b)=a\times \tilde{f}_{\uno,\, y,\, \uno}(\varepsilon_A,\, c,\, \varepsilon_A)\times b,
\end{displaymath}
 and $\tilde{f}_{\uno,\,y,\,\uno}$ is functorial in $y$ and linear in the variable $c$. Hence, if we consider the functors $\mathrm{Hom}_{A^e-\mathrm{Mod}}(A^{\otimes n+2},\,\_\,)$ and $\mathrm{Hom}_{\mathcal{F}}(A^{\otimes n},\,\_\,)$ from $A^e$-Mod to $R$-Mod, we obtain a natural isomorphism between them. Indeed, by sending $f$ to $\tilde{f}_{\uno,\,\_ \, ,\uno}(\varepsilon_A,\,\_ ,\,\varepsilon_A)$ we define an isomorphism of $R$-modules from $\mathrm{Hom}_{A^e-\mathrm{Mod}}(A^{\otimes n+2},\,N)$ to $\mathrm{Hom}_{\mathcal{F}}(A^{\otimes n},\,N)$, which is clearly natural in $N$. By precomposition with the functor $M_{\,\_\,}:\mathcal{X}\rightarrow A^e$-Mod, we obtain the result.
 
The second statement follows as in the classical case. That is,  since $\mu:A\otimes A\rightarrow A$ is an epimorphism, we have a monomorphism in $\mathcal{F}$,
\begin{displaymath}
\mathcal{H}_{A^e}(\mu, \, M): \mathcal{H}_{A^e}(A, \, M)\rightarrow \mathcal{H}_{A^e}(A^{\otimes 2}, \, M)\cong M
\end{displaymath}
defined in $x\in \mathcal{X}$ by sending $f\in \mathrm{Hom}_{A^e-\mathrm{Mod}}(A, \, M_x)$ to $M(\lambda_x)(f\circ\mu)_{\uno}\left((\varepsilon_{A^e})\right)$. By Remark \ref{mu-uni}  we have $\mu_{\uno}(\varepsilon_{A^e})=\varepsilon_A$, since $\varepsilon_{A^e}=\varepsilon_{A\otimes A}$. Hence $(f\circ\mu)_{\uno}(\varepsilon_{A^e})=f_{\uno}(\varepsilon_A)$. But, since $f$ is a morphism of $A$-bimodules, it is easy to see that $M(\lambda_x)(f_{\uno}(\varepsilon_A))$ must be in $CM_A(x)$. On the other hand,  given an element $n\in CM_A(x)$ and defining $f_y:A(y)\rightarrow M_x(y)$ by $a\mapsto a\times n$, we obtain a morphism of bimodules, $f$, such that $M(\lambda_x)(f_{\uno}(\varepsilon_A))=n$.
\end{proof}
 
With this, the complex $\mathcal{H}_{A^e}(C^{bar}_{\ast}(A)_{\mu},\, M)$ now looks like  
\[
\xymatrix{
0\ar[r]&CM_A\ar[r]& M  \ar[r]^-{\beta}&\mathcal{H}(A,\, M)  \ar[r]^-{\beta}&\mathcal{H}(A^{\otimes 2},\, M)  \ar[r]^-{\beta}& \ldots,
}
\]
with $M$ in degree zero. The second arrow from the left is just the inclusion of $CM_A$ in $M$. Now let $x$ be an object in $\mathcal{X}$.  We see immediately that   $\beta_x:M(x)\rightarrow \mathrm{Hom}_{\mathcal{F}}(A,\, M_x)$ is given by
\begin{displaymath}
\beta_x(m)_y:A(y)\rightarrow M(y\diamond x),\quad a\mapsto a\times m-M(s_{x\diamond y}^{y\diamond x})(m\times a).
\end{displaymath}
The rest of the arrows $\beta$ are described as follows. Let $n$ be an integer with $n\geq 1$. In the complex $\mathrm{Hom}_{A^e-\mathrm{Mod}}(C^{bar}_{\ast}(A),\, M_x)$, an arrow $\varphi$ in $\mathrm{Hom}_{A^e-\mathrm{Mod}}(A^{\otimes n+2},\, M_x)$ is sent to $\varphi\circ b$ in   $\mathrm{Hom}_{A^e-\mathrm{Mod}}(A^{\otimes n+3},\, M_x)$. We will see what this means in terms of the corresponding linear maps of Remark \ref{flechas}. 

Denote by $\tilde{\varphi}$ the $(n+2)$-linear map that corresponds to $\varphi$ by Remark \ref{flechas}. As in the previous lemma, the corresponding $n$-linear map     
\begin{displaymath}
\tilde{\varphi}_{\uno,x_1^2,\ldots, x_n^{n+1},\uno}: A(x_1)\times \cdots\times A(x_{n})\longrightarrow M_x(x_1\diamond\cdots\diamond x_{n})
\end{displaymath}
is given by $\tilde{\varphi}_{\uno,x_1^2,\ldots, x_n^{n+1},\uno}(\varepsilon_A,\,\_\,,\ldots,\,\_\,, \varepsilon_A)$. In the variables $x_j$ we add the superscript $i$ to denote the position of $x_j$ in $\tilde{\varphi}$. If $d_i:A^{\otimes n+3}\rightarrow A^{\otimes n+2}$, with $2\leq i \leq n+1$, is one of the morphisms defined in Section \ref{subsec-resbar}, then it is not hard to see that 
\begin{displaymath}
\widetilde{(\varphi\circ d_i)}_{\uno,x_1^2,\ldots, x_{n+1}^{n+2},\uno}: A(x_1)\times \cdots\times A(x_{n+1})\longrightarrow M_x(x_1\diamond\cdots\diamond x_{n+1})
\end{displaymath}
sends an element 
$(a_1, \ldots ,a_{n+1})$ to 
\begin{displaymath}
\tilde{\varphi}_{\uno,x_1^2,\ldots ,x_{i-1}^i\diamond x_i^{i+1},\ldots ,x_{n+1}^{n+1},\uno}(\varepsilon_A, a_1,\ldots ,a_{i-1}\times a_i,\ldots ,a_{n+1},\varepsilon_A).
\end{displaymath}
Also,  
\begin{displaymath}
\widetilde{(\varphi\circ d_1})_{\uno,x_1^2,\ldots, x_{n+1}^{n+2},\uno}(a_1,\ldots ,a_{n+1})=\tilde{\varphi}_{x_1^1,\ldots, x_{n+1}^{n+1},\uno}(a_1,\ldots ,a_{n+1},\, \varepsilon_A)
\end{displaymath}
and  
\begin{displaymath}
\widetilde{(\varphi\circ d_{n+2}})_{\uno,x_1^2,\ldots, x_{n+1}^{n+2},\uno}(a_1,\ldots ,a_{n+1})=\tilde{\varphi}_{\uno,x_1^2,\ldots, x_{n+1}^{n+2}}(\varepsilon_A,\, a_1,\ldots ,a_{n+1}). 
\end{displaymath}
Thus, if we let $f\in \mathrm{Hom}_{\mathcal{F}}(A^{\otimes n},\, M_x)$ to be the arrow corresponding to $\varphi$ and $\beta_x (f)$  be the arrow in $\mathrm{Hom}_{\mathcal{F}}(A^{\otimes n+1},\, M_x)$ corresponding to $\varphi\circ b$, then
\begin{displaymath}
\widetilde{\beta_x (f)}_{x_1,\ldots ,x_{n+1}}: A(x_1)\times \cdots\times A(x_{n+1})\longrightarrow M_x(x_1\diamond\cdots\diamond x_{n+1})
\end{displaymath}
sends $(a_1, \ldots ,a_{n+1})$ to 
\begin{displaymath}
\begin{array}{l}
a_1\times \tilde{f}_{x_2,\ldots ,x_{n+1}}(a_2,\ldots ,a_{n+1})\\
+\sum_{i=1}^n(-1)^i\tilde{f}_{x_1,\ldots ,x_i\diamond x_{i+1},\ldots ,x_{n+1}}(a_1,\ldots ,a_i\times a_{i+1},\ldots ,a_{n+1})\\
+(-1)^{n+1}M(s_{x_1\diamond \ldots \diamond x_n\diamond x\diamond x_{n+1}}^{x_1\diamond \ldots \diamond x_{n+1}\diamond x})(\tilde{f}_{x_1,\ldots ,x_n}(a_1,\ldots ,a_n)\times a_{n+1}),
\end{array}
\tag{$\dagger$}
\end{displaymath}
where the last symmetry comes from Remark \ref{modder}. The vanishing of  expression $\dagger$ is called a   cocycle condition and, in this case, the arrow $f$ is called a cocycle.

\section{Relative projectivity and separable monoids} 
\label{sec-relpro}

Given  a monoid $A$ in $\mathcal{F}$, we consider the functors
\begin{displaymath}
\mathcal{R}_A:A\textrm{-Mod}\rightarrow \mathcal{F}\quad \textnormal{and}\quad \mathcal{I}_A:\mathcal{F}\rightarrow A\textrm{-Mod},
\end{displaymath} 
the first one is just the forgetful functor and the second one sends a functor $F$ to $A\otimes F$ and an arrow $f$ to $A\otimes f$. It is easy to see that $\mathcal{I}_A$ is a left adjoint of $\mathcal{R}_A$.

Following Definition 4.1 in \cite{resmac} we say that an $A$-module $M$ is \textit{projective with respect to the restriction $\mathcal{R}_A$} if for any diagram
\[
\xymatrix{
 & M\ar[d]^-{\beta}\\
 X\ar[r]^-{\alpha}& Y
}
\]
in $A$-Mod such that $\mathcal{R}_A(\alpha)$ is a split epimorphism, there exists a morphism $\varphi:M\rightarrow X$ in $A$-Mod such that $\alpha\varphi=\beta$.

Now consider $\mathcal{R}_{A^e}$ and $\mathcal{I}_{A^e}$. Lemma 4.6 in \cite{resmac} translates in the following way.

\begin{lema}
\label{respro}
For every object $M$ in $A^e$-Mod there exists a resolution 
\[
\xymatrix{
\cdots\ar[r]&L_i\ar[r]&L_{i-1}\ar[r]&\cdots\ar[r]&L_0\ar[r]&M\ar[r]&0
}
\]
where the $L_i$ are projective with respect to the restriction $\mathcal{R}_{A^e}$ and such that
\[
\xymatrix{
\cdots\ar[r]&\mathcal{R}_{A^e}(L_i)\ar[r]&\mathcal{R}_{A^e}(L_{i-1})\ar[r]&\cdots\ar[r]&\mathcal{R}_{A^e}(L_0)\ar[r]&\mathcal{R}_{A^e}(M)\ar[r]&0
}
\]
is an exact split complex. Moreover, such a resolution for $M$ is unique up to homotopy.
\end{lema}

\begin{rem}
\label{barpro}
By point 3 of Lemma 4.3 in \cite{resmac}, every $A^{\otimes i}$ with $i\geq 2$, is projective with respect to $\mathcal{R}_A$. Also, by the proof of Proposition \ref{laresbar}, the restriction of the bar resolution, $\mathcal{R}_{A^e}(C^{bar}_{\ast}(A)_{\mu})$,
is an exact split complex. Hence, the bar resolution is a resolution for $A$ as in the previous lemma. 
\end{rem}

This remark will also help us to prove the following result.

\begin{lema}
\label{ses}
Let $E$ be a short exact sequence in $A^e$-Mod, 
\[
\xymatrix{
0\ar[r]&K\ar[r]^-j&M\ar[r]^-z&N\ar[r]&0
}
\]
such that $\mathcal{R}_{A^e}(z)$ is a split epimorphism. Then we have a short exact sequence of complexes in $\mathcal{F}$,
\[
\xymatrix{
0\ar[r]&\mathcal{H}_{A^e}(C^{bar}_{\ast}(A),\, K)\ar[r]&\mathcal{H}_{A^e}(C^{bar}_{\ast}(A),\, M)\ar[r]&\mathcal{H}_{A^e}(C^{bar}_{\ast}(A),\, N)\ar[r]&0.
}
\]
\end{lema}
\begin{proof}
Observe first that, for an object $x$ in $\mathcal{X}$, the shifted sequence $E_x$,
\[
\xymatrix{
0\ar[r]&K_x\ar[r]^-{j_x}&M_x\ar[r]^-{z_x}&N_x\ar[r]&0
}
\]
is also a short exact sequence in $A^e$-Mod. Also $(\mathcal{R}_{A^e}(L))_x=\mathcal{R}_{A^e}(L_x)$ and $(\mathcal{R}_{A^e}(l))_x=\mathcal{R}_{A^e}(l_x)$ for any $A$-bimodule $L$ and any morphism of $A$-bimodules $l$. Finally, if $\mathcal{R}_{A^e}(z)$ is a split epimorphism, then $\mathcal{R}_{A^e}(z_x)$ is also a split epimorphism.

It is easy to see that we have a sequence of complexes in $\mathcal{F}$,
\[
\xymatrix{
0\ar[r]&\mathcal{H}_{A^e}(C^{bar}_{\ast}(A),\, K)\ar[r]&\mathcal{H}_{A^e}(C^{bar}_{\ast}(A),\, M)\ar[r]&\mathcal{H}_{A^e}(C^{bar}_{\ast}(A),\, N)\ar[r]&0.
}
\]
Now, Let $i\geq 2$.  The hom functor $\mathcal{H}_{A^e}(A^{\otimes i},\,\_\,)$ is clearly left exact, so to show that the sequence
\[
\xymatrix{
0\ar[r]&\mathrm{Hom}_{A^e}(A^{\otimes i},\, K_x)\ar[r]&\mathrm{Hom}_{A^e}(A^{\otimes i},\, M_x)\ar[r]&\mathrm{Hom}_{A^e}(A^{\otimes i},\, N_x)\ar[r]&0
}
\]
(we abbreviate $\mathrm{Hom}_{A^e-\textrm{Mod}}$ as $\mathrm{Hom}_{A^e}$) is exact, it suffices to show that the morphism $\mathrm{Hom}_{A^e}(A^{\otimes i},\, z_x)$ from $\mathrm{Hom}_{A^e}(A^{\otimes i},\, M_x)$ to $\mathrm{Hom}_{A^e}(A^{\otimes i},\, N_x)$ is surjective. Let $f\in\mathrm{Hom}_{A^e}(A^{\otimes i},\, N_x)$. As in the previous remark, every $A^{\otimes i}$ with $i\geq 2$ is projective with respect to $\mathcal{R}_{A^e}$, so, by definition, there exists $\varphi: A^{\otimes i}\rightarrow M_x$ in $A^e$-Mod such that $z_x\circ \varphi=f_x$.
\end{proof}

We have immediately the following corollary.

\begin{coro}
\label{sula}
If $E$ is a short exact sequence in $A^e$-Mod as in the previous lemma, then we have a long exact sequence
\[
\xymatrix{
0\ar[r]&\mathcal{H}H^{0}(A,\, K)\ar[r]&\mathcal{H}H^0(A,\, M)\ar[r]&\mathcal{H}H^0(A,\, N)\ar[r]&\mathcal{H}H^1(A,\,K)\ldots 
}
\] 
in $\mathcal{F}$.
\end{coro}

Since for any $A$-bimodule $M$, the functor $\mathcal{H}_{A^e}(\,\_\,,\, M)$ is additive, we can calculate the Hochschild cohomology with coefficients in $M$ from any resolution that satisfies the conditions of Lemma \ref{respro}, as it is the case in the following example.

\begin{defi}
Let $A$ be a monoid in $\mathcal{F}$. We say that $A$ is separable if $\mu :A\otimes A\rightarrow A$ is a split epimorphism in $A^e$-Mod. 
\end{defi}

\begin{teo}
\label{teosep}
If $A$ is a separable monoid in $\mathcal{F}$, then for any $A$-bimodule $M$ we have $\mathcal{H}H^i(A,\, M)=0$ for every $i>0$. Conversely, if $A$ is a monoid in $\mathcal{F}$ such that $\mathcal{H}H^1(A,\, M)=0$ for any $A$-bimodule $M$, then $A$ is separable.
\end{teo}
\begin{proof}
If $A$ is separable, we have a split short exact sequence in $A^e$-Mod
\[
\xymatrix{
0\ar[r]&\textrm{Ker}\mu\ar[r]&A\otimes A\ar[r]^-{\mu}&A\ar[r]&0.
}
\]
As in Remark \ref{barpro}, $A\otimes A$ is projective with respect to the restriction $\mathcal{R}_{A^e}$ and by Lemma 4.2 in \cite{resmac}, we have that $\textrm{Ker}\mu$ is also projective with respect to $\mathcal{R}_{A^e}$. Hence, this short exact sequence gives a resolution for $A$ as in Lemma \ref{respro}. Since this resolution is contractible, we have that the bar resolution is contractible and thus, the complex $\mathcal{H}_{A^e}(C^{bar}_{\ast}(A)_{\mu},\, M)$ is contractible.

For the other direction, consider the short exact sequence in $A^e$-Mod,
\[
\xymatrix{
0\ar[r]&\textrm{Ker}\mu\ar[r]&A\otimes A\ar[r]^-{\mu}&A\ar[r]&0
}
\]
and let $I=\textrm{Ker}\mu$. By the first lines of the proof of Proposition  \ref{laresbar}, we have that $\mathcal{R}_{A^e}(\mu)$ is a split epimorphism. Then, by Corollary \ref{sula}, we have the following exact sequence in $\mathcal{F}$
\[
\xymatrix{
0\ar[r]&\mathcal{H}H^0(A,\, I)\ar[r]&\mathcal{H}H^0(A,\, A\otimes A)\ar[r]&\mathcal{H}H^0(A,\, A)\ar[r]& 0,
}
\]
where the  zero on the right corresponds to $\mathcal{H}H^1(A,\, I)=0$. But, by the last lines of the previous section, we have that $\mathcal{H}H^0(A,\, M)\cong CM_A$ for any $A$-bimodule $M$, hence, by Lemma \ref{deschom}, we have the following exact sequence in $R$-Mod,
\[
\xymatrix{
0\ar[r]&\mathcal{H}_{A^e}(A,\, I)(\uno)\ar[r]&\mathcal{H}_{A^e}(A,\, A\otimes A)(\uno)\ar[r]&\mathcal{H}_{A^e}(A,\, A)(\uno)\ar[r]& 0.
}
\]
So, for the identity morphism, $A$ in $\mathcal{H}_{A^e}(A,\, A)(\uno)$, there exists a morphism $f\in \mathcal{H}_{A^e}(A,\, A\otimes A)(\uno)$ such that $\mu\circ f=A$.
\end{proof}

We will finish the section with an example but, first, we need the following lemmas. The first one generalizes what happens in the classical case.

\begin{lema}
Let $A$ be a monoid in $\mathcal{F}$. Then $A$ is separable if and only if there exists an element $\xi\in C(A\otimes A)_{A}(\uno)$ such that $\mu_{\uno}(\xi)=\varepsilon$.  Moreover, $\xi$ determines the morphism of separability $t:A\rightarrow A\otimes A$. 
\end{lema}
\begin{proof}
Suppose first that $A$ is separable.  Then we have $t:A\rightarrow A\otimes A$ in $A^e$-Mod that satisfies $\mu\circ t=A$. Let $\xi=t_{\uno}(\varepsilon)$. 
For an object $x$ in  $\mathcal{X}$ and $a\in A(x)$, the action on the left gives
\begin{displaymath}
a\times \xi=a\times t_{\uno}(\varepsilon)=t_{x\diamond \uno}(a\times \varepsilon)=t_{x\diamond \uno}(A(\rho_x^{-1})(a))=(A\otimes A)(\rho_x^{-1}) (t_x(a)) 
\end{displaymath}
and on the right
\begin{displaymath}
\xi\times a= t_{\uno}(\varepsilon)\times a=t_{\uno\diamond x}(\varepsilon\times a)=t_{\uno\diamond x}(A(\lambda_x^{-1})(a))=(A\otimes A)(\lambda_x^{-1}) (t_x(a)), 
\end{displaymath}
but $\rho_x^{-1}=s_{\uno\diamond x}^{x\diamond \uno}\circ \lambda_x^{-1}$. This shows that $\xi\in C(A\otimes A)_{A}(\uno)$ and that it determines the morphism $t$.

Now suppose we have $\xi\in C(A\otimes A)_{A}(\uno)$ as in the hypothesis of the lemma and define $t:A\rightarrow A\otimes A$ in an object $x$ as
\begin{displaymath}
t_x(a):= (A\otimes A)(\rho_x)(a\times \xi),
\end{displaymath}
where $a\times \xi$ is the left action of $A$ on $A\otimes A$. Since $\xi\in C(A\otimes A)_{A}(\uno)$ this is equal to $(A\otimes A)(\lambda_x)(\xi\times a)$. This clearly defines a natural transformation from $A$ to $A\otimes A$ and the functoriality of the bi-action of $A$ on $A\otimes A$ shows that $t$ is a morphism of bimodules. Finally, we clearly have $t_{\uno}(\varepsilon)=\xi$ and thus $\mu_{\uno}\circ t_{\uno}(\varepsilon)=\varepsilon$. So, for any object $x$ in $\mathcal{X}$ and any $a\in A(x)$, since $\mu$ and $t$ are  morphisms of bimodules, we have
\begin{displaymath}
a=A(\rho_x)(\varepsilon\times a)=A(\rho_x)(\mu_{\uno \diamond x}t_{\uno\diamond x}(\varepsilon\times a))=\mu_{x}t_{x}(A(\rho_x)(\varepsilon\times a)),
\end{displaymath}
the last equality comes from the naturality of $\rho$. Thus $\mu\circ t=A$.
\end{proof}

\begin{prop}
Let $D$ and $A$ be monoids in $\mathcal{F}$ and suppose $f:D\rightarrow A$ is a morphism of monoids which is an epimorphism in $\mathcal{F}$. We have the following:
\begin{itemize}
\item[i)] If $D$ is separable, then $A$ is separable. 
\item[ii)] If $\mu_D$ is an isomorphism of $D$-bimodules, then $\mu_{A}$ is an isomorphism of $A$-bimodules.
\item[iii)] Suppose that $D$ is commutative and that $\mu_D$ is an isomorphism of $D$-bimodules. Then $A$ is commutative and $\mu_D$ and $\mu_A$ are isomorphisms of monoids. 
\end{itemize}

\end{prop}
\begin{proof}
Let us prove $i)$. We notice that $f$ (and we do not need it to be an epimorphism for this) induces a structure of $D$-bimodule in $A$, given on the left by $\mu_A\circ (f\otimes A)$ and on the right by $\mu_A\circ (A\otimes f)$.
Since $f$ is a morphism of monoids, it is easy to see it is also a morphism of $D$-bimodules. Then, $A\otimes A$ is a $D$-bimodule and $f\otimes f$ is a morphism of $D$-bimodules.

Let $\xi_D$ be an element in $C(D\otimes D)_{D}(\uno)$ as in the previous lemma and consider $\xi_A:=(f\otimes f)_{\uno}(\xi_D)$. We have
\begin{displaymath}
\mu_{A,\uno}(\xi_A)=\mu_{A,\uno}(f\otimes f)_{\uno}(\xi_D)=f_{\uno}\mu_{D,\uno}(\xi_D)=f_{\uno}(\varepsilon_D)=\varepsilon_A,
\end{displaymath}
since $f$ is a morphism of monoids.

Now let $x$ be  an object of $\mathcal{X}$ and $a\in A(x)$. Then there exists $b\in D(x)$ such that $f_x(b)=a$ and the left action of $a$ in $\xi_A$, $a\times \xi_A$, is equal to $f_x(b)\times (f\otimes f)_{\uno}(\xi_D)$. But this is precisely the left action of $D$ on $A\otimes A$ and, since $f\otimes f$ is a morphism of bimodules, we have that this is equal to
\begin{displaymath}
(f\otimes f)_{x\diamond \uno}(b\times \xi_D)=(f\otimes f)_{x\diamond \uno}\left((D\otimes D)(s_{\uno \diamond x}^{x\diamond \uno})(\xi_D\times b)\right),
\end{displaymath}
which, by functoriality, is equal to
\begin{displaymath}
(A\otimes A)(s_{\uno \diamond x}^{x\diamond \uno})((f\otimes f)_{\uno\diamond x}(\xi_D\times b)),
\end{displaymath}
which is equal to $(A\otimes A)(s_{\uno \diamond x}^{x\diamond \uno})(\xi_A\times a)$. Hence $\xi_A\in C(A\otimes A)_{A}(\uno)$ and $A$ is separable.

To prove $ii)$ suppose that $\mu_D$ has an inverse $t_D: D\rightarrow D\otimes D$ of $D$-bimodules. By the previous lemma, there exists an element $\xi_D\in C(D\otimes D)_D(\uno)$ such that $t_D=(D\otimes D)(\rho)(\,\_\,\times \xi_D)$. By point $i)$, if $\xi_A=(f\otimes f)_{\uno}(\xi_D)$, then $t_A:A\rightarrow A\otimes A$, given by $t_A=(A\otimes A)(\rho)(\,\_\,\times \xi_A)$, is a morphism of $A$-bimodules satisfying $\mu_At_A=A$. Let us see that $t_A\mu_A=A\otimes A$.

Let $x$ be an object in $\mathcal{X}$ and $b\in (A\otimes A)(x)$. Since $f$ is an epimorphism in $\mathcal{F}$, then $f\otimes f$ is also an epimorphism in $\mathcal{F}$. Hence, there exists $d\in (D\otimes D)(x)$ such that $(f\otimes f)_x(d)=b$. Thus
\begin{displaymath}
\begin{array}{rcl}
t_{A,\, x}\mu_{A,\, x}(b)& = & t_{A,\, x}\mu_{A,\, x}((f\otimes f)_x(d))\\
 & = & t_{A,\, x}(f_x\mu_{D,\, x}(d)) \quad \textrm{since f is a morphism of monoids}\\
 &=&(A\otimes A)(\rho_x)(f_x\mu_{D,\, x}(d)\times \xi_A)\\
 &=&(A\otimes A)(\rho_x)\left((f\otimes f)_{x\diamond \uno}(\mu_{D,\, x}(d)\times \xi_D) \right)\quad \textrm{as in point $i)$.}
\end{array}
\end{displaymath}
But $\mu_{D,\, x}(d)\times \xi_D$ is equal to $(D\otimes D)(\rho_x^{-1})(t_{D,\, x}\mu_{D,\, x}(d))$ and $t_{D,\, x}\mu_{D,\, x}(d)=d$, since $t_D\mu_D=D\otimes D$. Hence, $\mu_{D,\, x}(d)\times \xi_D=d\times \varepsilon_{D\otimes D}$, where the last product is the product in $D\otimes D$. Finally, since $f\otimes f$ is a morphism of monoids,
\begin{displaymath}
(f\otimes f)_{x\diamond \uno}(d\times \varepsilon_{D\otimes D})=(f\otimes f)_x(d)\times (f\otimes f)_{\uno}(\varepsilon_{D\otimes D})
\end{displaymath}
and so $t_{A,\, x}\mu_{A,\, x}(b)=(A\otimes A)(\rho_x)((f\otimes f)_x(d)\times \varepsilon_{A\otimes A})=b$.

Now suppose that $D$ es commutative. This means that $\mu_D\circ S=\mu_D$. Hence, 
\begin{displaymath}
\mu_A\circ(f\otimes f)=f\circ\mu_D=f\circ\mu_D\circ S=\mu_A\circ (f\otimes f)\circ S,
\end{displaymath}
the first equality on the left follows since $f$ is morphism of monoids. On the other hand, $(f\otimes f)\circ S=S\circ (f\otimes f)$. So, since $f\otimes f$ is an epimorphism, we have $\mu_A\circ S=\mu_A$ and $A$ is commutative. To finish $iii)$ it suffices to notice that $\mu$ becomes a morphism of monoids if the monoid is commutative. 
\end{proof}

\begin{coro}
 For very object $F$ in $\mathcal{F}$ and every $i>0$, we have $\mathcal{H}H^i(I,\, F)=0$. Also, if $A$ is a monoid  such that there exists a morphism of monoids $I\rightarrow A$ which is an epimorphism in $\mathcal{F}$, then $A$ is commutative and $A$-bimodules coincide with (left) $A$-modules. In particular, for every $A$-module $M$ and every $i>0$ we have $\mathcal{H}H^i(A,\, M)=0$.
\end{coro}

\begin{ejem}
The previous corollary implies that for every biset functor $F$ and every $i>0$, we have $\mathcal{H}H^i(RB,\, F)=0$. Also, if we take $R=k$ as a field of characteristic $0$ and $A=kR_{\rac}$ as the biset functor of rational representations, we know that $kR_{\rac}$ is a monoid and that the linearization morphism $kB\rightarrow kR_{\rac}$ provides an arrow in $\mathcal{F}_{\mathcal{C},\, R}$ as in the previous corollary. Hence, for every $kR_{\rac}$-module $M$ and every $i>0$, we have $\mathcal{H}H^i(kR_{\rac},\, M)=0$.
\end{ejem}

\section{Hochschild cohomology functors} 

Through this section $A$ is a monoid in $\mathcal{F}$,  $M$ and $N$ are $A$-bimodules and $x$, $y$, $z$ and $w$ are objects of $\mathcal{X}$.

Let $P$, $Q$ and $T$ be objects in $\mathcal{F}$. To simplify the notation, in what follows we will work indistinctly with $\mathrm{Hom}_{\mathcal{F}}(P\otimes Q,\, T)$ and  the corresponding set of bilinear maps given by Remark \ref{flechas}. In particular, given an arrow and $f:P\otimes Q\rightarrow T$ in $\mathcal{F}$ we  will avoid the notation $\tilde{f}$, used in Section \ref{subsec-comple}. 

\subsection{$n=0$}

As said in Lemma \ref{deschom}, in this case we clearly have
\begin{displaymath}
\mathcal{H}H^0(A,\, M)\cong CM_A\cong \mathcal{H}_{A^e}(A,\, M).
\end{displaymath}

\subsection{$n=1$}

\begin{defi}
An arrow $f\in \mathrm{Hom}_{\mathcal{F}}(A,\, M_x)$ is called a derivation from $A$ to $M_x$ if for  any $a\in A(y)$ and $b\in A(z)$, we have
\begin{displaymath}
f_{y\diamond z}(a\times b)=(a\times f_z(b))+M(s_{y\diamond x\diamond z}^{y\diamond z\diamond x})(f_y(a)\times b).
\end{displaymath}
The arrow $f$ is called an inner derivation if there exist $m\in M(x)$ such that the morphisms $f_y$ are given by sending $a\in A(y)$ to $a\times m- M(s_{x\diamond y}^{y\diamond x})(m\times a)$.
 \end{defi}

Using expression $\dagger$ from Section \ref{subsec-comple}, we obtain that the kernel of
\begin{displaymath}
\beta_x:\mathrm{Hom}_{\mathcal{F}}(A,\, M_x)\rightarrow \mathrm{Hom}_{\mathcal{F}}(A^{\otimes 2},\, M_x)
\end{displaymath}
corresponds to the derivations from $A$ to $M_x$ and that the image of $\beta_x:M(x)\rightarrow \mathrm{Hom}_{\mathcal{F}}(A,\, M_x)$ corresponds to the inner derivations from $A$ to $M_x$. In particular, we can define two subfunctors of $\mathcal{H}(A,\, M)$ in $\mathcal{F}$. First, that of derivations $\textrm{Der}(A,\, M)$, defined in $x$ as the derivations from  $A$ to $M_x$, and then that of inner derivations $\textrm{Inn}(A,\, M)$, defined accordingly. With this
\begin{displaymath}
\mathcal{H}H^1(A,\, M)=\frac{\textrm{Der}(A,\, M)}{\textrm{Inn}(A,\, M)}.
\end{displaymath}

As in the classical case, we can endow $\mathcal{H}H^1(A,\, A)$ with a Lie-type structure in the following way. First, we denote by $d\mapsto \overline{d}$ the projection map from $\textrm{Der}(A,\, M)(x)$ to $\mathcal{H}H^1(A,\, M)(x)$. Then, taking $\overline{d}\in \mathcal{H}H^1(A,\, A)(x)$ and $\overline{d'}\in \mathcal{H}H^1(A,\, A)(x')$, consider the  arrow $[d,\, d']:A\rightarrow A_{x\diamond x'}$, given  as
\begin{displaymath}
[d,\, d']_y=(d'_{y\diamond x}\circ d_y)-A(s_{y\diamond x'\diamond x}^{y\diamond x\diamond x'})(d_{y\diamond x'}\circ d'_y).
\end{displaymath}
Defining $[\overline{d},\, \overline{d'}]$ as   $\overline{[d, \, d']}\in \mathcal{H}H^1(A,\, A)(x\diamond x')$ is well defined and the  bracket satisfies the corresponding Lie-type axioms.

\begin{rem}
Serge Bouc has suggested that if  $A$ is a Green biset functor, then we can call $\mathcal{H}H^1(A,\, A)$ a \textit{Lie biset functor}.
\end{rem}

\subsection{$n=2$}

\begin{defi}
A square-zero extension of $A$ consists of a monoid $E$ in $\mathcal{F}$ and a monoid homomorphism $\pi:E\rightarrow A$ which is an epimorphism in $\mathcal{F}$ and such that $\textrm{Ker}\pi$ is an ideal of $E$ of \textit{square zero}, that is, the bilinear maps
\begin{displaymath}
\textrm{Ker}\pi(y)\times \textrm{Ker}\pi(z)\rightarrow \textrm{Ker}\pi(y\diamond z)
\end{displaymath}
are zero for all $y$ and $z$.

In this case, it is easy to see that $\textrm{Ker}\pi$ becomes an $A$-bimodule. If $\textrm{Ker}\pi$ is isomorphic to $N$, as $A$-bimodules, we call this extension a square-zero extension of $A$ by $N$.
\end{defi}

\begin{defi}
A square-zero extension of $A$ is called a Hochschild extension if it is a split extension in $\mathcal{F}$, i.e., there exists an arrow $\sigma:A\rightarrow E$ in $\mathcal{F}$ such that $\pi \sigma=A$.
\end{defi}

Two Hochschild extensions, $\pi: E\rightarrow A$ and $\pi':E'\rightarrow A$, of $A$ by $N$ are said to be equivalent if there exists a monoid morphism $\varphi: E\rightarrow E'$ such that the squares in the following diagram
\[
\xymatrix{
0\ar[r]&N\ar[d]^-N\ar[r]&E\ar[d]^-{\varphi}\ar[r]^-{\pi}&A\ar[d]^-A\ar[r]&0\\
0\ar[r]&N\ar[r]&E'\ar[r]^-{\pi'}&A\ar[r]&0
}
\]
commute. This defines an equivalence relation in the collection of all the Hochschild extensions of $A$ by $N$, we denote the collection of equivalence classes by $\mathrm{Ext}(A,\, N)$.

Notice that $\mathrm{Ext}(A,\, N)$ is not empty, since we have the class of the semidirect product $N\rtimes A$. As an object in $\mathcal{F}$, it is given by the coproduct of $N$ with $A$, that is $(N\rtimes A)(y)=N(y)\oplus A(y)$. As a monoid, the product is given by
\begin{displaymath}
((m_1,\, a_1), (m_2,\, a_2))\mapsto((m_1\times a_2)+(a_1\times m_2),\, a_1\times a_2),
\end{displaymath}
for $(m_1,\, a_1)\in N(y)\oplus A(y)$ and $(m_2,\, a_2)\in N(z)\oplus A(z)$.

\begin{teo}
There exists a bijection
\begin{displaymath}
\mathcal{H}H^2(A,\, M)(x)\leftrightarrow \mathrm{Ext}(A,\, M_x).
\end{displaymath} 
\end{teo}
\begin{proof}
By the cocycle condition of Section \ref{subsec-comple}, we have that the kernel of
\begin{displaymath}
\beta_x:\mathrm{Hom}_{\mathcal{F}}(A^{\otimes 2},\, M_x)\rightarrow \mathrm{Hom}_{\mathcal{F}}(A^{\otimes 3},\, M_x)
\end{displaymath}
consists of arrows $f\in \mathrm{Hom}_{\mathcal{F}}(A^{\otimes 2},\, M_x)$ that satisfy that for any $(a_1,\, a_2,\, a_3)$ in $A(y)\times A(z)\times A(w)$, the expression
\begin{displaymath}
a_1\times f_{z,\, w}(a_2,\, a_3)-f_{y\diamond z,\, w}(a_1\times a_2,\, a_3)+ f_{y,\, z \diamond w}(a_1,\, a_2\times a_3)-M(y\diamond z\diamond s_{x \diamond w}^{w\diamond x})(f_{y,\, z}(a_1,\, a_2)\times a_3)
\end{displaymath}
is equal to zero.

For such an $f$, we consider $E_f=M_x\oplus A$ in $\mathcal{F}$ with the following product
\begin{displaymath}
((m_1,\, a_1),(m_2, a_2))\mapsto (M(s_{y\diamond x\diamond z}^{y\diamond z\diamond x})(m_1\times a_2)+(a_1\times m_2)+f_{y,\, z}(a_1,\, a_2),\, a_1\times a_2),
\end{displaymath}
for $(m_1,\, a_1)\in M_x(y)\oplus A(y)$ and $(m_2,\, a_2)\in M_x(z)\oplus A(z)$. The equality to zero of the expression above allows us to show that this product is associative. Also, it shows that if we take $m_{\varepsilon}=-f_{\uno,\, \uno}(\varepsilon,\, \varepsilon)$, then $(m_{\varepsilon},\, \varepsilon)\in M_x(\uno)\oplus A(\uno)$ is the identity element for the product. Finally, it is easy to verify that this product is functorial. With this, it is immediate to see that the extension
\[
\xymatrix{
0\ar[r]&M_x\ar[r]^-{\iota}&E_f\ar[r]^-{\varrho}&A\ar[r]&0,
}
\]
where $\iota$ is the inclusion morphism and $\varrho$ is the projection morphism, is a Hochschild extension. Now consider $f'=f-\beta_x(g)$, where $g\in \mathrm{Hom}_{\mathcal{F}}(A, M_x)$ and 
\begin{displaymath}
\beta_x: \mathrm{Hom}_{\mathcal{F}}(A,\, M_x)\rightarrow \mathrm{Hom}_{\mathcal{F}}(A^{\otimes 2},\, M_x).
\end{displaymath}
Then, the extensions given by $E_f$ and $E_{f'}$ are equivalent. Indeed, the isomorphism of monoids $\varphi:E_f\rightarrow E_{f'}$ is given by sending $(m,\, a)\in M_x(y)\oplus A(y)$ to $(m+g_y(a),\, a)\in M_x(y)\oplus A(y)$.

On the other direction, given a Hochschild extension
\[
\xymatrix{
0\ar[r]&M_x\ar[r]^-{\gamma}&E\ar[r]^-{\pi}&A\ar[r]&0,
}
\]
since there exists $\sigma:A\rightarrow E$ in $\mathcal{F}$ such that $\pi\sigma=A$, then $E\cong M_x\oplus A$ in $\mathcal{F}$. Hence, we can obtain the monoid structure of $E$ from this direct sum,
\begin{displaymath}
(\gamma_y(m_1),\, \sigma_y(a_1))\times (\gamma_z(m_2),\, \sigma_z(a_2)),
\end{displaymath}
for $(m_1,\, a_1)\in (M_x\oplus A)(y)$ and $(m_2,\, a_2)\in (M_x\oplus A)(z)$. To simplify the notation, let $n_1=\gamma_y(m_1)$, $b_1=\sigma_y(a_1)$, $n_2=\gamma_z(m_2)$ and  $b_2=\sigma_z(a_2)$. Then, since the product is bilinear and $M_x$ is of square zero, it is easy to see that the product above is equal to
\begin{displaymath}
\left(\textrm{Im}\gamma(s_{y\diamond x\diamond z}^{y\diamond z\diamond x})(n_1\times b_2)+(b_1\times n_2)+f^0_{y,\, z}(b_1,\, b_2), b_1\times b_2 \right),
\end{displaymath}
where $f^0_{y\diamond z}(b_1,\, b_2)$ is the unique element in $\textrm{Im}\gamma(y\diamond z)$  such that $(0,\, b_1)\times (0,\, b_2)=(f^0_{y\diamond z}(b_1,\, b_2),\, b_1\times b_2)$. This allows us to define a map $f^E_{y,\, z}: A(y)\times A(z)\rightarrow M_x(y\diamond z)$ sending $(a_1,\, a_2)$ to $\gamma_{y\diamond z}^{-1}f^0_{y\diamond z}(\sigma_y(a_1),\, \sigma_z(a_2))$. In turn, this defines an arrow  $f^E\in\mathrm{Hom}_{\mathcal{F}}(A^{\otimes 2},\, M_x)$ and the associativity of the product of $E$ implies that $f^E$ is in the kernel of $\beta_x$. Finally, suppose we have an equivalent extension
\[
\xymatrix{
0\ar[r]&M_x\ar[r]^-{\gamma'}&E'\ar[r]^-{\pi'}&A\ar[r]&0,
}
\]
with splitting morphism $\sigma':A\rightarrow E'$. Let $\varphi:E\rightarrow E'$ be the equivalence morphism. Then, as before, for $a_1\in A(y)$ there exist a unique $f^0_y(\sigma_y(a_1))\in \mathrm{Im}\gamma'(y)$ that corresponds to $\varphi_y(0,\, \sigma_y(a_1))\in E'(y)$. Hence we can define $g_y:A(y)\rightarrow M_x(y)$ by sending $a_1$ to $(\gamma')^{-1}_yf^0_y(\sigma_y(a_1))$. This defines an arrow in $\mathrm{Hom}_{\mathcal{F}}(A,\,M_x)$. Also, using the fact that $\varphi$ is a morphism of monoids, if we take the morphism defined before $f^{E'} \in\mathrm{Hom}_{\mathcal{F}}(A^{\otimes 2},\, M_x)$ corresponding to $E'$, then we obtain $f^E-f^{E'}=\beta_x(g)$. 
\end{proof}

\begin{coro}
$\mathrm{Ext}(A,\, M_{\,\_\,})$ is a functor in $\mathcal{F}$.  
\end{coro}
\begin{proof}
Indeed, since $\mathcal{H}H^2(A,\, M)(x)$ is an $R$-module, the previous bijection endows $\mathrm{Ext}(A,\, M_x)$ with a structure of $R$-module. The functorial structure is also given by the  bijection. That is, given an arrow  $\alpha:x\rightarrow y$ in $\mathcal{X}$, then the map
\begin{displaymath}
\mathrm{Ext}(A,\, M_x)\rightarrow \mathrm{Ext}(A,\, M_y)
\end{displaymath}
sends (the class of) the extension $E_f$, corresponding to the cocycle $f$, to $E_{M_{\alpha}\circ f}$.
\end{proof}

\begin{rem}
It is not hard to see that the sum we obtain in $\mathrm{Ext}(A,\, M_x)$, by the previous corollary, corresponds to the Baer sum of extensions. Given two Hochschild extensions, $\pi:E\rightarrow A$ and $\pi':E'\rightarrow A$,  of $A$ by $N$, their Baer sum is defined in $y$ as the quotient $\Pi (y)/\Gamma (y)$, where $\Pi (y)$ is the pullback of $\pi_y$ and $\pi'_y$, that is
\begin{displaymath}
\Pi (y)=\{(a,\, b)\in E(y)\oplus E'(y)\mid \pi_y(a)=\pi'_y(b)\},
\end{displaymath}
and  $\Gamma (y)=\{(\gamma_y(n),\,-\gamma'_y(n))\mid n\in N(y)\}$, if $\gamma:N\rightarrow E$ and $\gamma':N\rightarrow E'$ are the corresponding morphisms. It is defined in an obvious way in arrows. This construction yields a Hochschild extension of $A$ by $N$.
\end{rem}

\subsection{The Hochschild cohomology monoid}

\begin{defi}
An $\mathbb{N}$\textit{-graded monoid} in $\mathcal{F}$ is a  lax  monoidal functor $G:\mathbb{N}\rightarrow \mathcal{F}$, where $\mathbb{N}$ is seen as a discrete monoidal category with monoidal structure given by the addition in $\mathbb{N}$. 
\end{defi}

That is, an $\mathbb{N}$-graded monoid consists of the following:
\begin{itemize}
\item For each $i\in \mathbb{N}$, a functor $F_i$ in $\mathcal{F}$.
\item For every $i,\, j\in \mathbb{N}$ and every $x$ and $y$ objects in $\mathcal{X}$, a bilinear map
\begin{displaymath} 
F_i(x)\times F_j(y)\rightarrow F_{i+j}(x\diamond y),
\end{displaymath}
that is associative in an obvious way and functorial in $x$ and $y$ (analogous conditions to those appearing in Section \ref{homfun} for a monoid).
\item An element $\varepsilon_{F_0}\in F_0(\uno)$ such that  for every $i\in \mathbb{N}$ and $a\in F_i(x)$,
\begin{displaymath}
a=F_i(\gamma_x)(\varepsilon_{F_0}\times a)=F_i(\rho_x)(a\times \varepsilon_{F_0}).
\end{displaymath}
\end{itemize}

In this case, $F_0$ is a monoid in $\mathcal{F}$ and every $F_i$ is an $F_0$-bimodule.

Since $\mathcal{F}$ has arbitrary coproducts, we can consider  
\begin{displaymath}
\mathcal{H}H(A):=\bigoplus_{i\in \mathbb{N}}\mathcal{H}H^i(A,\, A)
\end{displaymath}
in $\mathcal{F}$. We will endow $\mathcal{H}H(A)$ with a structure of graded monoid. We begin by defining, for $i,\, j\in \mathbb{N}$, a cup product,
\begin{displaymath}
\_\,\smile\,\_:\mathcal{H}(A^{\otimes i},\, A)(x)\times \mathcal{H}(A^{\otimes j},\, A)(y)\rightarrow \mathcal{H}(A^{\otimes i+j},\, A)(x\diamond y),
\end{displaymath}
sending $(f^i,\, f^j)$, with $f^i\in \textrm{Hom}_{\mathcal{F}}(A^{\otimes i},\, A_x)$ and $f^j\in \textrm{Hom}_{\mathcal{F}}(A^{\otimes j},\, A_y)$, to 
\begin{displaymath}
f^i\smile f^j: A(x_1)\times\cdots\times A(x_i)\times A(y_1)\times\cdots\times A(y_j)\rightarrow A_{x\diamond y}(x_1\diamond\cdots\diamond x_i\diamond y_1\diamond\cdots \diamond y_j),
\end{displaymath}
which sends an $(i+j)$-tuple $(a_1,\ldots ,a_i,\, b_1,\ldots , b_j)$ to 
\begin{displaymath}
A\left(s_{(i)\diamond x\diamond (j)\diamond y}^{(i+j)\diamond x\diamond y}\right)\left( f^i_{x_1,\ldots ,x_i}(a_1,\ldots ,a_i)\times f^j_{y_1,\ldots ,y_j}(b_1,\ldots ,b_j)\right),
\end{displaymath}
here $s_{(i)\diamond x\diamond (j)\diamond y}^{(i+j)\diamond x\diamond y}$ is the symmetry in $\mathcal{X}$ sending $x_1\diamond\cdots\diamond x_i\diamond x\diamond y_1\diamond\cdots \diamond y_j\diamond y$ to $x_1\diamond\cdots\diamond x_i\diamond  y_1\diamond\cdots \diamond y_j\diamond x\diamond y$.

This product makes $\bigoplus_{i=0}^n \mathcal{H}(A^{\otimes i},\, A)$ an  $\mathbb{N}$-graded monoid in $\mathcal{F}$, with identity $\varepsilon_A\in A(\uno)\cong \mathcal{H}(A^{\otimes 0},\, A)(\uno)$. Next, consider the complex $\mathcal{H}(A^{\otimes *},\, A)$ with morphisms $\beta$ as in Section 4.2. A series of straightforward computations (in particular one must pay attention to the symmetries that appear in the expressions) show that for any $i,\, j\in \mathbb{N}$ and $f^i$ and $f^j$ as before,
\begin{displaymath}
\beta (f^i\smile f^j)=\beta f^i\smile f^j+(-1)^if^i\smile \beta f^j.
\end{displaymath}
This formula allows us to extend the cup product to $\bigoplus_{i\in \mathbb{N}}\mathcal{H}H^i(A,\, A)$, just as in the classical case. Also, it is easy to see that the product satisfies the conditions for a graded monoid, with  $\varepsilon_A\in CA(\uno)\cong \mathcal{H}H^0(A,\, A)(\uno)$.

\begin{rem}
If $A$ is a Green biset functor, then we can call $\mathcal{H}H(A)$ a \textit{graded Green biset functor}.
\end{rem}

\begin{rem}
The cup product just described works exactly the same if we consider first
\begin{displaymath}
\_\,\smile\,\_: A(x)\times \mathcal{H}(A^{\otimes j},\, M)(y)\rightarrow \mathcal{H}(A^{\otimes j},\, M)(x\diamond y), 
\end{displaymath}
via the isomorphism $A(x)\cong \mathcal{H}(A^{\otimes 0},\, A)(x)$ and then
\begin{displaymath}
\_\,\smile\,\_: CA(x)\times \mathcal{H}H^j(A,\, M)(y)\rightarrow \mathcal{H}H^j(A,\, M)(x\diamond y). 
\end{displaymath}
This shows that each $\mathcal{H}H^j(A,\, M)$ is a $CA$-module.
\end{rem}

\subsection{Further results}

The author believes that the following results, not considered in this paper, can also be extended to our context.

\begin{itemize}
\item[i)] The description of $\mathcal{H}H^3(A,\, M)$ in terms of \textit{crossed bimodules} (see E.1.5.1 in \cite{loday}).
\item[ii)] The definition of a bracket giving $\mathcal{H}H(A)$ a structure of \textit{graded Lie monoid}, with the grading shifted by -1, as in the classical case. 
\end{itemize}

\section*{Acknowledgments}

The contents of this article were developed during a sabbatical year the author did at the laboratory LAMFA of the Universit\'e de Picardie, in Amiens, France, from August 2022 to July 2023. The author thanks the staff and colleagues at LAMFA for all the support she received during her stay,  which led to the successful development of the project. Special thanks to Serge Bouc, host during the sabbatical, for all the ideas, suggestions and stimulating conversations.  


\end{document}